\documentclass[a4paper,11pt]{article}

\usepackage{amsmath,amsthm,amssymb,mathrsfs,latexsym,amsfonts}
\usepackage{graphicx,psfrag,epsfig}
\usepackage[english]{babel}
\usepackage[latin1]{inputenc}
\usepackage{pst-all}

\newtheorem{theorem}{Theorem}[section]
\newtheorem{lemma}[theorem]{Lemma}
\newtheorem{proposition}[theorem]{Proposition}

\newcounter{paraga}[section]
\renewcommand{\theparaga}{{\bf\arabic{paraga}.}}
\newcommand{\paraga}{\medskip \addtocounter{paraga}{1} 
\noindent{\theparaga\ } }

\begin{document}

\bibliographystyle{amsalpha}

\def\MP{\,{<\hspace{-.5em}\cdot}\,}
\def\SP{\,{>\hspace{-.3em}\cdot}\,}
\def\PM{\,{\cdot\hspace{-.3em}<}\,}
\def\PS{\,{\cdot\hspace{-.3em}>}\,}
\def\EP{\,{=\hspace{-.2em}\cdot}\,}
\def\PP{\,{+\hspace{-.1em}\cdot}\,}
\def\PE{\,{\cdot\hspace{-.2em}=}\,}
\def\N{\mathbb N}
\def\C{\mathbb C}
\def\Q{\mathbb Q}
\def\R{\mathbb R}
\def\T{\mathbb T}
\def\A{\mathbb A}
\def\Z{\mathbb Z}
\def\demi{\frac{1}{2}}
\def\eps{\varepsilon}
\def\pfrac#1#2{{\textstyle{#1\over#2}}}
\def\abs#1{\vert #1\vert}

\begin{titlepage}
\author{Abed Bounemoura}
\title{\LARGE{\textbf{Nekhoroshev estimates for finitely differentiable quasi-convex Hamiltonians}}}
\end{titlepage}

\maketitle

\begin{abstract}
A major result about perturbations of integrable Hamiltonian systems is the Nekhoroshev theorem, which gives exponential stability for all solutions provided the system is analytic and the integrable Hamiltonian is generic. In the particular but important case where the latter is quasi-convex, these exponential estimates have been generalized by Marco and Sauzin if the Hamiltonian is Gevrey regular, using a method introduced by Lochak in the analytic case. In this paper, using the same approach, we investigate the situation where the Hamiltonian is assumed to be only finitely differentiable, for which it is known that exponential stability does not hold but nevertheless we prove estimates of polynomial stability. 
\end{abstract}

\section{Introduction} \label{s1}

In this paper, we are concerned with the stability properties of near-integrable Hamiltonian systems of the form
\begin{equation*}
\begin{cases} 
H(\theta,I)=h(I)+f(\theta,I) \\
|f| < \varepsilon <\!\!<1
\end{cases}
\end{equation*}
where $(\theta,I) \in \T^n \times \R^n$ are action-angle coordinates for the integrable part $h$ and $f$ is a small perturbation in some suitable topology defined by a norm $|\,.\,|$. More precisely, we are interested in the evolution of the action variables $I(t)$, which are trivially constant in the absence of perturbation.

\paraga The first main result in this direction is given by an application of the KAM theory on the persistence of quasi-periodic solutions (see \cite{Pos01} for a recent exposition). Assuming $h$ satisfies some non-degeneracy condition and the system is analytic, if $\varepsilon$ is sufficiently small there exists a constant $c$ such that
\[ |I(t)-I_0| \leq c\sqrt\varepsilon, \quad t\in\R, \] 
for ``most" initial actions $I_0$, more precisely for a set of large measure but with empty interior. When $n=2$, this is even true for all solutions provided $h$ is isoenergetically non degenerate, but for $n>3$, the famous example of Arnold (\cite{Arn64}) shows that there exist ``unstable" solutions, along which the variation of the actions can be arbitrarily large no matter how small the perturbation is. From its very beginning, KAM theory was known to hold for non-analytic Hamiltonians (see \cite{Mos62} in the context of twist maps). It is now well established in various regularity classes, including the $C^{\infty}$ case (essentially by Herman, see \cite{Bos86} and \cite{Fej04}) and the Gevrey case (\cite{Pop04}). Following ideas of Moser, the theorem also holds if $H$ is only of class $C^k$, with $k>2n$ (see \cite{Pos82}, \cite{Sal04}, \cite{SZ89} and also \cite{Alb07} for a refinement), even though the minimal number of derivatives is still an open question, except in a special case for $n=2$ (\cite{Her86}).

\paraga Another fundamental result, which complements KAM theory, is given by Nekhoroshev's theorem (\cite{Nek77}, \cite{Nek79}). If the integrable part $h$ satisfies some generic condition and the system is analytic, then for $\varepsilon$ sufficiently small there exist positive constants $c_1,c_2,c_3,a$ and $b$ such that
\[ |I(t)-I_0| \leq c_1\varepsilon^b, \quad |t|\leq c_2\exp(c_3\varepsilon^{-a}), \]   
for all initial actions $I_0$. Hence all solutions are stable, not for all time, but for an exponentially long time. In the special case where $h$ is strictly quasi-convex, a completely new proof of these estimates was given by Lochak (\cite{Loc92}) using periodic averaging and simultaneous Diophantine approximation. The method of Lochak has had many applications, in particular it was used by Marco and Sauzin to extend Nekhoroshev's theorem to the Gevrey regular case under the quasi-convexity assumption (\cite{MS02}). 

\paraga However, no such estimates have been studied when the Hamiltonian is merely finitely differentiable, and this is the content of the present paper. We will prove below (Theorem~\ref{theoremeCk1}) that if $H$ is of class $C^k$, for $k\geq 2$, and $h$ quasi-convex, then one has the stability estimates   
\[|I(t)-I_0| \leq c_1\varepsilon^{\frac{1}{2n}}, \quad |t|\leq c_2\varepsilon^{-\frac{k-2}{2n}},\]
for some positive constants $c_1$ and $c_2$, and provided that $\varepsilon$ is small enough. Of course, under our regularity assumption the exponential estimates have been replaced with polynomial estimates, and earlier examples show that exponential stability cannot possibly hold under such a weak regularity assumption (this is discussed in \cite{MS04}). The proof will use once again the ideas of Lochak which, among other things, reduces the analytic part to its minimum and we will also follow the implementation of Marco and Sauzin in the Gevrey case.

\paraga As we recalled above, KAM theory for finitely differentiable Hamiltonian systems has been widely studied, and so we believe that Nekhoroshev's estimates under weaker regularity assumptions have their own interest. Moreover, for obvious reasons, examples of unstable solutions (so-called Arnold diffusion) are more easily constructed in the non-analytic case, and it is a natural question to estimate the speed of instability (see \cite{KL08a} and \cite{KL08b} for examples of class $C^k$ with a polynomial speed of diffusion). Finally, one of our motivations is to generalize these estimates using the method of \cite{BN09}, where Lochak's ideas are extended to deal with analytic but $C^k$-generic unperturbed Hamiltonians, with $k>2n+2$.  

\section{Main result} \label{s2}

\paraga Let $\T^n=\R^n/\Z^n$, and consider a Hamiltonian function $H$ defined on the domain 
\[ \mathcal{D}_R=\T^n \times B_R,\] 
where $B_R$ is the open ball of $\R^n$ around the origin of radius $R$, with respect to the supremum norm $|\,.\,|$. As usual, we shall occasionally identify $H$ with a function defined on $\R^n\times B_R$ which is $1$-periodic with respect to the first $n$ variables. 

We assume that $H$ is of class $C^k$, for an integer $k\geq 2$, {\it i.e.} it is $k$-times differentiable and all its derivatives up to order $k$ extend continuously to the closure $\overline{\mathcal{D}}_R$. We denote by $C^k(\mathcal{D}_R)$ the space of such functions, which is a Banach space with the norm
\[|H|_{C^k(\mathcal{D}_R)}=\sup_{0\leq l\leq k}\sup_{|\alpha|=l}\left(\sup_{x\in\overline{\mathcal{D}}_R}|\partial^\alpha H(x)|\right)\] 
where $x=(\theta,I)$, $\alpha=(\alpha_1,\dots,\alpha_{2n})\in\N^{2n}$, $|\alpha|=\alpha_1+\cdots+\alpha_{2n}$ and 
\[ \partial^\alpha=\partial_{1}^{\alpha_1} \dots \partial_{2n}^{\alpha_{2n}}.\] 
In the case where the Hamiltonian $H=h$ depends only on the action variables, we will simply write $|h|_{C^k(B_R)}$.  

Our Hamiltonian $H \in C^k(\mathcal{D}_R)$ is assumed to be $C^k$-close to integrable, that is, of the form
\begin{equation}\label{HamCk}
\begin{cases} \tag{$\ast$}
H(\theta,I)=h(I)+f(\theta,I) \\
|f|_{C^k(\mathcal{D}_R)}<\varepsilon <\!\!< 1
\end{cases}
\end{equation}
where $h$ is the integrable part and $f$ a small perturbation of size $\varepsilon$ in the $C^k$ topology. 

We may assume that there exists a positive constant $M$ such that 
\begin{equation} \label{bounded}
|h|_{C^k(B_R)}<M. \tag{B}
\end{equation} 
Moreover, we will require that $h$ satisfy the following quasi-convexity assumption: there exists a positive constant $m$ such that
\begin{equation} \label{quasiconv}
\forall I\in B_R, \forall v\in\R^n, \quad \nabla h(I).v=0 \Longrightarrow \nabla^2 h(I)v.v \geq m|v|^2. \tag{C}
\end{equation} 

\paraga Our main theorem is the following.

\begin{theorem}\label{theoremeCk1}
Let $H$ be as in~(\ref{HamCk}), with $k\geq 3$, $h$ satisfying~(\ref{bounded}) and~(\ref{quasiconv}), and set
\[ a=\frac{k-2}{2n}, \quad b=\frac{1}{2n}.\] 
Then there exist $\varepsilon_0$, $c_1$ and $c_2$ such that if $\varepsilon\leq\varepsilon_0$, all solutions $(\theta(t),I(t))$ of $H$ with $I(0)\in B_{R/2}$ satisfy
\[ |I(t)-I(0)|\leq c_1\varepsilon^b, \quad |t| \leq c_2 \varepsilon^{-a}. \]
\end{theorem}

First note that we have stated our theorem for $H$ of class $C^k$, $k\geq3$, but the statement also holds with no changes if $H$ is of class $C^{k-1,1}$, $k\geq 3$, that is, $H$ is of class $C^{k-1}$ and its partial derivatives of order $k-1$ are Lipschitz continuous. For $C^2$ systems, these estimates also hold true but it is not useful since trivially all solutions satisfy
\[ |I(t)-I(0)|\leq \varepsilon, \quad |t| \leq 1. \] 
In fact for $C^2$ Hamiltonians for which the derivatives up to order $2$ are not more than continuous, we believe that one cannot obtain better estimates.

Moreover, the time of stability obtained is ``optimal" in the sense that one can construct examples of unstable orbits with a polynomial speed of diffusion, but we do not know what the optimal exponents should be. However, using the geometric arguments of \cite{BM10} one can easily improve the stability exponent $a$ in order to obtain
\[ a=\frac{k-2}{2(n-1)}-\delta, \] 
for $\delta>0$ but arbitrarily small.

Let us finally point out that if $H$ is $C^\infty$, then it is an immediate consequence of the above result that the action variables are stable for an interval of time which is longer than any prescribed power of $\varepsilon^{-1}$, but even in this case exponential stability does not hold.

\paraga As in the analytic or Gevrey case, we can also state a refined result near resonances. Suppose $\Lambda$ is a sub-module of $\Z^n$ of rank $m$, $d=n-m$ and let $S_{\Lambda}$ be the corresponding resonant manifold, that is
\[ S_{\Lambda}=\{I \in B_R \; | \; k.\nabla h(I)=0, \; k \in \Lambda\}. \] 
We can prove the following statement, which actually contains the previous one.

\begin{theorem}\label{theoremeCk2}
Under the previous hypotheses, assume $d(I(0),S_{\Lambda}) \leq \sigma\sqrt{\varepsilon}$ for some constant $\sigma>0$, and set
\[ a_d=\frac{k-2}{2d}, \quad b_d=\frac{1}{2d}.\]
Then there exist $\varepsilon_0'$, $c_1'$ and $c_2'$ such that if $\varepsilon\leq\varepsilon_0'$, one has 
\[ |I(t)-I(0)|\leq c_1' \varepsilon^{b_d}, \quad |t| \leq c_2'\varepsilon^{-a_d}.  \]
\end{theorem}

For $\Lambda=\{0\}$, $d=n$ and $S_\Lambda=B_{R/2}$, we recover Theorem~\ref{theoremeCk1} and therefore it will be enough to prove Theorem~\ref{theoremeCk2}. 

\paraga The constants $\varepsilon_0$, $c_1$ and $c_2$ depend only on $h$, more precisely they depend on $k,n,R,M$ and $m$ while the constants $\varepsilon_0'$, $c_1'$ and $c_2'$ also depend on $\sigma$ and $\Lambda$. However we will not give explicit values for them in order to avoid complicated and rather meaningless expressions. Hence we shall replace them by the symbol $\cdot$ when it is convenient: for instance, we shall write $u \MP v$ when there exists a positive constant $c$ depending only on the previous parameters, but not on $f$, such that $u<cv$.

\paraga This paper is divided into two sections. The next section contains the analytical part of the proof, where we will construct a system of local coordinates for our Hamiltonian which is more convenient to study the evolution of the action variables. Then, in the last section we will conclude the proof using our convexity assumption and Dirichlet's theorem on simultaneous Diophantine approximation. 

\section{Analytical part} \label{s3}

\paraga Given an action $I\in B_R$ and denoting by $\omega=\nabla h(I)$ its frequency, we know from classical averaging theory that the relevant part of the perturbation 
\[ f(\theta,I)=\sum_{k\in\Z^n}\hat{f}_k(I)e^{i2\pi k.\theta}\] 
is given by those harmonics associated with integers $k\in\Z^n$ in resonance with $\omega$, that is such that $k.\omega=0$. Actually one can construct a symplectic, close-to-identity transformation $\Phi$ defined around $I$, such that 
\[ H\circ\Phi=h+g+\tilde{f} \]
where $g$ contains only harmonics in resonance with $\omega$ and $\tilde{f}$ is a small remainder. These are usually called resonant normal forms, and to obtain them one has to deal with small divisors $k.\omega$ which involve technical estimates. If the system is analytic, the above remainder $\tilde{f}$ can be made exponentially small with respect to the inverse of the size of the perturbation, as was first shown by Nekhoroshev. But for finitely differentiable systems one might guess that the remainder can only be polynomially small, even though this should be difficult (or at least technical) to prove using the usual approach.

\paraga It is a remarkable fact discovered by Lochak (\cite{Loc92}) that to prove exponential estimates in the quasi-convex case with the analyticity assumption, it is enough to average along periodic frequencies, which are frequencies $\omega$ such that $T\omega \in \Z^n\setminus\{0\}$ for some $T>0$ (see also \cite{BN09} for an extension of this method for generic integrable Hamiltonians). These periodic frequencies correspond to periodic orbits of the unperturbed Hamiltonian, hence in this approach no small divisors arise. As a consequence this special resonant normal form is much easier to obtain. The aim of this section is to construct such a normal form, up to a polynomial remainder. This will be done in~\ref{FormCk}. But first we will recall some useful estimates concerning the $C^k$ norm in~\ref{PrelE}, and then prove an intermediate statement in~\ref{LinCase}. 

\subsection{Elementary estimates}\label{PrelE}

\paraga Let us begin by recalling some easy estimates. Given two functions $f,g\in C^k(\mathcal{D}_R)$, the product $fg$ belongs to $C^k(\mathcal{D}_R)$ and by the Leibniz rule
\[ |fg|_{C^k(\mathcal{D}_R)} \MP |f|_{C^k(\mathcal{D}_R)}|g|_{C^k(\mathcal{D}_R)}. \] 
The Poisson Bracket $\{f,g\}$ belongs to $C^{k-1}(\mathcal{D}_R)$, and by its definition and the Leibniz rule one gets
\[ |\{f,g\}|_{C^{k-1}(\mathcal{D}_R)} \MP |f|_{C^k(\mathcal{D}_R)}|g|_{C^k(\mathcal{D}_R)}. \]
The above implicit constants depend only on $n$ and $k$ (in fact in the first estimate one can trivially modify the definition of the $C^k$ norm so as to have a constant equal to one, but this will not be important for us). These are very elementary facts, but we shall also need estimates concerning vector fields, canonical transformations and compositions. 

\paraga First, given a vector-valued function $F\in C^k(\mathcal{D}_R,\R^l)$, $F=(F_1,\dots,F_l)$ and $l\in\N$, we extend the norm component-wise, that is
\[ |F|_{C^k(\mathcal{D}_R)}=\sup_{1\leq i \leq l}|F_i|_{C^k(\mathcal{D}_R)}. \]
Now for a function $f\in C^k(\mathcal{D}_R)$, we define its Hamiltonian vector field $X_f$ by 
\[ X_f=(\partial_I f,-\partial_\theta f) \]
where
\[ \partial_I f=(\partial_{I_1} f,\dots,\partial_{I_n} f), \quad \partial_\theta f=(\partial_{\theta_1} f,\dots,\partial_{\theta_n} f). \]
Obviously $X_f\in C^{k-1}(\mathcal{D}_R,\R^{2n})$, and trivially
\[ |X_f|_{C^{k-1}(\mathcal{D}_R)} \leq |f|_{C^k(\mathcal{D}_R)}. \]
Moreover, by classical theorems on ordinary differential equations, if $X_f$ is of class $C^{k-1}$ then so is the time-$t$ map $\Phi_{t}^{f}$ of the vector field $X_f$, when it exists. Assuming $|\partial_\theta f|_{C^{0}(\mathcal{D}_R)}<r$ for some $r<R$ (for example $|X_f|_{C^{0}(\mathcal{D}_R)}<r$), then by the mean value theorem 
\[ \Phi^f=\Phi_{1}^{f} : \mathcal{D}_{R-r} \longrightarrow \mathcal{D}_{R} \]
is a well-defined $C^{k-1}$-embedding. In the case where $f$ is integrable, one can choose $r=0$.

In the sequel, we will need to estimate the $C^{k}$ norm of $\Phi^f$ in terms of the $C^{k}$ norm of the vector field $X_f$. More precisely we need the rather natural fact that $\Phi^f$ is $C^{k}$-close to the identity when $X_f$ is $C^{k}$-close to zero. This is trivial for $k=0$. In the general case, this follows by induction on $k$ using on the one hand the relation 
\[ \Phi_{t}^{f}=\mathrm{Id}+\int_{0}^{t}X_f \circ \Phi_{s}^{f} ds, \]
and on the other the formula of Faà di Bruno (see \cite{AR67} for example), which gives bounds of the form
\[ |F\circ G|_{C^k} \MP |F|_{C^k}|G|^{k}_{C^k} \]
and also 
\[ |F\circ G|_{C^k} \MP |F|_{C^1}|G|^{k}_{C^k}+|F|_{C^k}|G|^{k}_{C^{k-1}} \]
for $C^k$ vector-valued functions on appropriate domains (once again, the above implicit constants depend only on $k$). Let us state this as a lemma, for which we refer to \cite{DH09}, Lemma $3.15$ and appendix $C$, for a detailed proof. 

\begin{lemma}\label{estimflot} 
Let $X_f\in C^{k}(\mathcal{D}_R,\R^{2n})$, assume that $|X_f|_{C^{0}(\mathcal{D}_R)}<r$ and
\begin{equation}\label{smallnCk}
|X_f|_{C^{k}(\mathcal{D}_R)} < 1.  
\end{equation}
Then 
\[ |\Phi^f-\mathrm{Id}|_{C^{k}(\mathcal{D}_{R-r})} \MP |X_f|_{C^{k}(\mathcal{D}_R)}.  \]
\end{lemma} 

The above implicit constant depends only on $k$ and $R$. Now if $H\in C^{k}(\mathcal{D}_{R})$, under the above hypotheses we have $H\circ \Phi^f\in C^{k}(\mathcal{D}_{R-r})$ and the estimate
\begin{equation}\label{estimcomp}
|H\circ\Phi^f|_{C^{k}(\mathcal{D}_{R-r})} \MP |H|_{C^{k}(\mathcal{D}_{R})}|\Phi^f|_{C^{k}(\mathcal{D}_{R-r})}^{k}  
\end{equation} 
follows trivially from the Faà di Bruno formula. 

\subsection{The linear case}\label{LinCase}

Following \cite{MS02}, we change for a moment our setting and we consider a perturbation of a linear Hamiltonian, more precisely the Hamiltonian
\begin{equation}\label{HamLCk}
\begin{cases} \tag{$\ast\ast$}
H(\theta,I)=l(I)+f(\theta,I) \\
|f|_{C^k(\mathcal{D}_\rho)}<\mu <\!\!< 1
\end{cases}
\end{equation}  
where $\rho>0$ is fixed and $l(I)=\omega.I$ is a linear Hamiltonian with a $T$-periodic frequency $\omega$. Recall that this means that 
\[ T=\inf\{t>0 \; | \; t\omega \in \Z^n \setminus \{0\}\}\] 
is well-defined. In this context, our small parameter is $\mu$.

In the proposition below, we will construct a ``global" normal form for the Hamiltonian~(\ref{HamLCk}), which we will use in the next section to produce a ``local" normal form around periodic orbits for our original Hamiltonian~(\ref{HamCk}).

\begin{proposition} \label{lemmehamLCk}
Consider $H$ as in~(\ref{HamLCk}) with $k\geq 2$, and assume 
\begin{equation}\label{smallmu}
T\mu\MP 1, \quad |\omega|\MP 1.
\end{equation}
Then there exists a $C^2$ symplectic transformation 
\[ \Phi : \mathcal{D}_{\rho/2} \rightarrow \mathcal{D}_\rho\] 
with $|\Phi-\mathrm{Id}|_{C^{2}(\mathcal{D}_{\rho/2})}\MP T\mu$ such that 
\begin{equation*}
H \circ \Phi=l+g+f,
\end{equation*}
with $\{g,l\}=0$ and the estimates
\begin{equation*}
|g|_{C^{2}(\mathcal{D}_{\rho/2})} \MP \mu, \quad |f|_{C^{2}(\mathcal{D}_{\rho/2})} \MP (T\mu)^{k-2}\mu 
\end{equation*}
hold true.
\end{proposition}

First note that $\{g,l\}=0$ means exactly that $\partial_\theta g.\omega=0$, and expanding $g$ in a Fourier series, one easily sees that it contains only harmonics associated with integers $k$ satisfying $k.\omega=0$. Therefore the above proposition gives indeed a resonant normal form, up to a polynomial remainder.
 
Note also that we need our transformed Hamiltonian $H\circ\Phi$, and hence our transformation $\Phi$, to be at least of class $C^2$, simply because we need our transformed vector field to be of class $C^1$ to have existence and uniqueness of solutions (or course, a $C^{1,1}$ regularity would have been enough). This explains the factor $(k-2)$ in our stability exponent $a$. 

\begin{proof}
First, as we are assuming $|\omega|\MP 1$, we have 
\begin{equation}\label{omegalin}
|\Phi_t^l|_{C^k(\mathcal{D}_\rho)} \MP 1.
\end{equation} 
Our transformation $\Phi$ will be obtained by a finite composition of averaging transformations. Let us define 
\[ r=\frac{\rho}{2(k-2)}\]
and for $j\in\{0,\dots,k-2\}$, let  
\[ \rho_j=\rho-jr\geq \rho/2. \] 
Then we claim that for any $j\in\{0,\dots,k-2\}$, there exists a $C^{k-j}$ symplectic transformation $\Phi_j : \mathcal{D}_{\rho_j} \rightarrow \mathcal{D}_\rho$ with $|\Phi_j-\mathrm{Id}|_{C^{k-j}(\mathcal{D}_{\rho_j})}\MP T\mu$ such that 
\begin{equation*}
H \circ \Phi_j=l+g_j+f_j,
\end{equation*}
with $g_j$ and $f_j$ of class $C^{k-j}$, $\{g_j,l\}=0$ and the estimates
\begin{equation*}
|g_j|_{C^{k-j}(\mathcal{D}_{\rho_j})} \MP \mu, \quad |f_j|_{C^{k-j}(\mathcal{D}_{\rho_j})} \MP (T\mu)^{j}\mu. 
\end{equation*}
The proposition follows easily by taking $\Phi=\Phi_{k-2}$, $g=g_{k-2}$ and $f=f_{k-2}$. 

We will prove the claim by induction on $j\in\{0,\dots,k-2\}$. For $j=0$, there is nothing to prove since we can simply write $H=l+g_0+f_0$ with $g_0=0$, $f_0=f$ and therefore $\Phi_0$ is the identity. Now assume the claim is true for some $j\in\{0,\dots,k-3\}$, and consider 
\[ H_j=H\circ\Phi_j=l+g_j+f_j.\] 
Let us define
\[ [f_j]=\frac{1}{T}\int_{0}^{T}f_j \circ \Phi_{t}^{l}dt\] 
and
\[\chi_j=\frac{1}{T}\int_{0}^{T}t(f_j-[f_j])\circ \Phi_{t}^{l}dt.\]

We have 
\[ |[f_j]|_{{C^{k-j}}(\mathcal{D}_{\rho_j})} \leq |f_j \circ \Phi_{t}^{l}|_{{C^{k-j}}(\mathcal{D}_{\rho_j})}  \]
hence by~(\ref{estimcomp}) and~(\ref{omegalin}) we obtain  
\[ |[f_j]|_{{C^{k-j}}(\mathcal{D}_{\rho_j})} \MP |f_j|_{{C^{k-j}}(\mathcal{D}_{\rho_j})}  \]
and with our hypotheses of induction, this gives
\[ |[f_j]|_{{C^{k-j}}(\mathcal{D}_{\rho_j})} \MP (T\mu)^{j}\mu.  \]
Similarly 
\[ |\chi_j|_{{C^{k-j}}(\mathcal{D}_{\rho_j})} \MP T(T\mu)^{j}\mu=(T\mu)^{j+1} \]
and since $T\mu \MP 1$,
\[ |\chi_j|_{{C^{k-j}}(\mathcal{D}_{\rho_j})} \MP T\mu. \]

If we let $\Phi^{\chi_j}$ be the time-one map of the Hamiltonian vector field generated by $\chi_j$, then we will show that the map 
\[\Phi_{j+1}=\Phi_j \circ \Phi^{\chi_j}\] 
satisfies the assumptions. 

Indeed, thanks to the condition $T\mu \MP 1$, we can ensure that $\Phi^{\chi_j}$, which is of class $C^{k-j-1}$, is a well-defined embedding 
\[ \Phi^{\chi_j} : \mathcal{D}_{\rho_{j+1}} \longrightarrow \mathcal{D}_{\rho_j}.\] 
Moreover, as $|X_{\chi_j}|_{{C^{k-j-1}}(\mathcal{D}_{\rho_j})} \MP T\mu$ and using once again $T\mu \MP 1$, we can arrange condition~(\ref{smallnCk}) and apply Lemma~\ref{estimflot} to obtain
\[|\Phi^{\chi_j}-\mathrm{Id}|_{C^{k-j-1}(\mathcal{D}_{\rho_{j+1}})}\MP T\mu.\]
Now 
\begin{eqnarray*}
|\Phi_{j+1}-\mathrm{Id}|_{C^{k-j-1}(\mathcal{D}_{\rho_{j+1}})} & = & |\Phi_j \circ \Phi^{\chi_j} - \Phi^{\chi_j} + \Phi^{\chi_j} -\mathrm{Id}|_{C^{k-j-1}(\mathcal{D}_{\rho_{j+1}})}  \\
& \leq & |(\Phi_j - \mathrm{Id})\circ \Phi^{\chi_j}|_{C^{k-j-1}(\mathcal{D}_{\rho_{j+1}})} \\
& + & |\Phi^{\chi_j} -\mathrm{Id}|_{C^{k-j-1}(\mathcal{D}_{\rho_{j+1}})}  \\
& \MP & |\Phi_j - \mathrm{Id}|_{C^{k-j-1}(\mathcal{D}_{\rho_{j}})} + |\Phi^{\chi_j} -\mathrm{Id}|_{C^{k-j-1}(\mathcal{D}_{\rho_{j+1}})}  
\end{eqnarray*}
where we have used~(\ref{estimcomp}) in the last line. By our hypotheses of induction, this eventually gives
\[ |\Phi_{j+1}-\mathrm{Id}|_{C^{k-j-1}(\mathcal{D}_{\rho_{j+1}})}\MP T\mu. \] 

Now by Taylor's formula with integral remainder, we can expand
\[ H_{j+1}=H \circ \Phi_{j+1}=l+g_{j+1}+f_{j+1} \]
with 
\begin{equation*}
g_{j+1}=g_{j}+[f_j], \quad f_{j+1}=\int_{0}^{1}\{g_{j}+f_{j}^{t},\chi_j\}\circ \Phi_{t}^{\chi}dt 
\end{equation*}
where $f_j^t=tf_j+(1-t)[f_j]$, as one can check by a standard calculation. Since $\{g_{j},l\}=0$ by our hypothesis of induction and obviously $\{[f_{j}],l\}=0$, we have $\{g_{j+1},l\}=0$ together with the estimate
\begin{eqnarray*}
|g_{j+1}|_{C^{k-j}(\mathcal{D}_{\rho_{j}})} & \leq  & |g_j|_{C^{k-j}(\mathcal{D}_{\rho_j})} + |[f_j]|_{C^{k-j}(\mathcal{D}_{\rho_j})} \\
& \MP & \mu + (T\mu)^{j}\mu  \\
& \MP & \mu 
\end{eqnarray*}
using $T\mu \MP 1$. For the remainder, using~(\ref{estimcomp}) and the trivial estimate for the Poisson bracket we compute
\begin{eqnarray*}
|f_{j+1}|_{C^{k-j-1}(\mathcal{D}_{\rho_{j+1}})} & \leq & |\{g_{j}+f_{j}^{t},\chi_j\}\circ \Phi_{t}^{\chi_j}|_{C^{k-j-1}(\mathcal{D}_{\rho_{j+1}})} \\
& \MP & |\{g_{j}+f_{j}^{t},\chi_j\}|_{C^{k-j-1}(\mathcal{D}_{\rho_{j}})} \\
& \MP & |\{g_j,\chi_j\}|_{C^{k-j-1}(\mathcal{D}_{\rho_{j}})} + |\{f_j^t,\chi_j\}|_{C^{k-j-1}(\mathcal{D}_{\rho_{j}})}  \\
& \MP & |g_j|_{C^{k-j}(\mathcal{D}_{\rho_{j}})}|\chi_j|_{C^{k-j}(\mathcal{D}_{\rho_{j}})}+|f_j^t|_{C^{k-j}(\mathcal{D}_{\rho_{j}})}|\chi_j|_{C^{k-j}(\mathcal{D}_{\rho_{j}})}  \\
& \MP & \mu(T\mu)^{j+1}+\mu(T\mu)^j(T\mu)^{j+1} \\
& \MP & \mu(T\mu)^{j+1}
\end{eqnarray*}
using once again $T\mu \MP 1$. This concludes the proof.
\end{proof}

\subsection{Normal form}\label{FormCk}

Now let us come back to our original setting which is the Hamiltonian 
\begin{equation*}
\begin{cases}
H(\theta,I)=h(I)+f(\theta,I) \\
|f|_{C^k(\mathcal{D}_R)}<\varepsilon <\!\!< 1.
\end{cases}
\end{equation*}
We will say that an action $I_*\in B_R$ is $T$-periodic if its frequency vector $\nabla h(I_*)$ is $T$-periodic. In the proposition below, we will fix a $T$-periodic action $I_*$, $l$ will be the linear integrable Hamiltonian $l(I)=\omega.I$ associated with the periodic frequency $\omega=\nabla h(I_*)$ and we denote by 
\[ \Pi_I : \T^n \times B_R \rightarrow B_R\] 
the projection onto the action space. 

\begin{proposition}[Normal form]\label{lemmehamCk}
Suppose $H$ is as in~(\ref{HamCk}), with $h$ satisfying~(\ref{bounded}). Under the previous hypotheses, let $\mu>0$ be such that 
\begin{equation}\label{small2}
\varepsilon \MP \mu^2, \quad \mu \MP 1, \quad T\mu \MP 1.
\end{equation}
Then there exists a $C^2$ symplectic transformation 
\[\Phi :\T^n \times B(I_*,\mu)\rightarrow \T^n \times B(I_*,2\mu)\] 
with $|\Pi_I\Phi-\mathrm{Id}_I|_{C^{0}(B(I_*,\mu))}\MP T\mu^2$ such that 
\begin{equation*}
H \circ \Phi=h+g+\tilde{f},
\end{equation*}
with $\{g,l\}=0$ and the estimates
\begin{equation*}
|g+\tilde{f}|_{C^{0}(\T^n \times B(I_*,\mu))} \MP \mu^2, \quad |\partial_{\tilde{\theta}} \tilde{f}|_{C^{0}(\T^n \times B(I_*,\mu))} \MP (T\mu)^{k-2}\mu^2
\end{equation*} 
hold true.
\end{proposition}

Let us immediately explain how such coordinates will be used in the sequel. If we denote them by $(\tilde{\theta},\tilde{I})$, then writing down the equations of motion for $\tilde{H}=H\circ\Phi$, using the fact that $\partial_{\tilde{\theta}} g.\omega=0$ and the mean value theorem, one shows that $\tilde{I}(t)$ remains close to the hyperplane orthogonal to $\omega$, passing through $\tilde{I}(0)$, for an interval of time governed by the size of $\partial_{\tilde{\theta}} \tilde{f}$. Hence any potential drift has to occur along that hyperplane. Then this picture persists and gets only slightly distorted when we come back to the original coordinates $(\theta,I)$, since the projection of $\Phi$ onto action space is close to identity. 

\begin{proof}
First note that since $h$ satisfies~(\ref{bounded}), then
\[ |\omega|=|\nabla h(I_*)|<M \]
and therefore the second part of condition~(\ref{smallmu}) is satisfied.

Without loss of generality, we may assume $h(I_*)=0$. To analyze our Hamiltonian $H$ in a neighborhood of size $\mu$ around $I^*$, we translate and rescale the action variables using the map
\[ \sigma_\mu : (\theta,I^\mu) \longmapsto (\theta,I)=(\theta,I_*+\mu I^\mu) \]
which sends the domain $\mathcal{D}_{2}=\T^n \times B_{2}$ onto $\T^n \times B(I_*,2\mu)$, and note that by the condition $\mu \MP 1$, we can assume that the latter domain is included in $\mathcal{D}_R$. Let 
\[ H_\mu=\mu^{-1}(H\circ\sigma_\mu)\] 
be the rescaled Hamiltonian, so $H_\mu$ is defined on $\mathcal{D}_{2}$ and reads
\[ H_\mu(\theta,I^\mu)=\mu^{-1}H(\theta,I_*+\mu I^\mu)=\mu^{-1}h(I_*+\mu I^\mu)+\mu^{-1}f(\theta,I_*+\mu I^\mu) \]
for $(\theta,I^\mu)\in\mathcal{D}_{2}$. Now using Taylor's formula we can expand $h$ around $I_*$ to obtain
\begin{eqnarray*}
h(I_*+\mu I^\mu) & = & \mu\,\omega.I^\mu+\mu^2\int_{0}^{1}(1-t)\nabla^2 h(I_*+t\mu I^\mu)I^\mu.I^\mu dt \\
& = & \mu\,\omega.I^\mu+\mu^{2}h_\mu(I^\mu) 
\end{eqnarray*}
where we set
\[ h_\mu(I^\mu)=\int_{0}^{1}(1-t)\nabla^2 h(I_*+t\mu I^\mu)I^\mu.I^\mu dt. \]
Therefore we can write
\[ H_\mu=l+f_\mu \]
with 
\[ f_\mu=\mu h_\mu+\mu^{-1}(f\circ\sigma_\mu). \]
Now we know that $|f|_{C^k(\mathcal{D}_R)}<\varepsilon \MP \mu^2$ from the first part of condition~(\ref{small2}) and $|h|_{C^k(\mathcal{D}_R)}\MP 1$ since we are assuming~(\ref{bounded}), so we obviously have
\[ |f_\mu|_{C^k(\mathcal{D}_{2})} \MP \mu. \]   
Therefore conditions~(\ref{smallmu}) are satisfied and we can apply Proposition~\ref{lemmehamLCk}, with $\rho=2$, to the Hamiltonian $H_\mu=l+f_\mu$: there exists a $C^2$ symplectic transformation $\Phi_\mu : \mathcal{D}_{1} \rightarrow \mathcal{D}_{2}$ with $|\Phi_\mu-\mathrm{Id}|_{C^{2}(\mathcal{D}_{1})}\MP T\mu$ such that 
\begin{equation*}
H_\mu \circ \Phi_\mu=l+g_\mu+\tilde{f}_\mu
\end{equation*}
with $\{g_\mu,l\}=0$ and the estimates
\begin{equation*}
|g_\mu|_{C^{2}(\mathcal{D}_{1})} \MP \mu, \quad |\tilde{f}_\mu|_{C^{2}(\mathcal{D}_{1})} \MP (T\mu)^{k-2}\mu. 
\end{equation*}
Moreover, if we set
\[ \hat{f}_\mu=g_\mu-\mu h_\mu,\] 
we have $\{\hat{f}_\mu,l\}=0$ and $|\hat{f}_\mu|_{C^{2}(\mathcal{D}_{1})}\MP \mu$, and so the transformed Hamiltonian can also be written as
\begin{equation*}
H_\mu \circ \Phi_\mu=l+\mu h_\mu+\hat{f}_\mu+\tilde{f}_\mu.
\end{equation*}
Now scaling back to our original coordinates, we define $\Phi=\sigma_{\mu} \circ \Phi_\mu \circ \sigma_{\mu}^{-1}$, therefore
\[ \Phi : \T^n\times B(I_*,\mu) \longrightarrow \T^n\times B(I_*,2\mu)\] 
and 
\begin{eqnarray*}
H\circ\Phi & = & \mu H_\mu\circ\Phi_\mu \circ \sigma_{\mu}^{-1} \\
& = & \mu (l+\mu h_\mu+\hat{f}_\mu+\tilde{f}_\mu) \circ \sigma_{\mu}^{-1} \\
& = & (\mu l+\mu^2 h_\mu) \circ \sigma_{\mu}^{-1} + \mu\hat{f}_\mu\circ \sigma_{\mu}^{-1} + \mu\tilde{f}_\mu\circ \sigma_{\mu}^{-1}.  
\end{eqnarray*}
Observe that $(\mu l+\mu^2 h_\mu) \circ \sigma_{\mu}^{-1}=h$, so we may set
\[ g=\mu\hat{f}_\mu\circ \sigma_{\mu}^{-1}, \quad \tilde{f}=\mu\tilde{f}_\mu\circ \sigma_{\mu}^{-1}, \]
and write
\[ H\circ\Phi=h+g+\tilde{f}. \]
It is obvious that $\{g,l\}=0$ with 
\[|g|_{C^{0}(\T^n \times B(I_*,\mu))}\leq\mu |\hat{f}_\mu|_{C^{0}(\mathcal{D}_{1})} \MP \mu^2\]
and similarly
\[ |\tilde{f}|_{C^{0}(\T^n \times B(I_*,\mu))}\leq\mu |\tilde{f}_\mu|_{C^{0}(\mathcal{D}_{1})} \MP (T\mu)^{k-2}\mu^2 \MP \mu^2\]
so 
\[ |g+\tilde{f}|_{C^{0}(\T^n \times B(I_*,\mu))} \MP \mu^2. \]
Moreover, as $\partial_{\tilde{\theta}} \tilde{f}=\mu \partial_{\tilde{\theta}} \tilde{f}_\mu$ then
\[ |\partial_{\tilde{\theta}} \tilde{f}|_{C^{0}(\T^n \times B(I_*,\mu))} \MP (T\mu)^{k-2}\mu^2\]
and finally
\[ |\Pi_I\Phi-\mathrm{Id}_I|_{C^{0}(B(I_*,\mu ))} \MP T\mu^2. \]
is trivial. This ends the proof.
\end{proof}

\section{Proof of Theorem~\ref{theoremeCk2}} \label{s4}

\paraga Now we can complete the proof of our Theorem~\ref{theoremeCk2} in the spirit of Lochak, following three elementary steps that use successively some arithmetic (simultaneous Diophantine approximation), some analysis (normal forms around periodic orbits) and some geometry (quasi-convexity). The analysis has been done in the previous section, and the arithmetic and geometry are exactly the same as in the analytic case or the Gevrey case. Therefore instead of rewriting proofs which are well-known, we will merely explain the ideas and state the relevant results which can be found in~\cite{Loc92} and~\cite{MS02}.

\paraga Let us begin with the arithmetic part, since in order to use our Proposition~\ref{lemmehamCk} (the normal form), we will need to show that any action $I_0 \in B_{R/2}$, which is close to some resonant surface $S_{\Lambda}$, can be approximated by a periodic action. Let $\omega_0=\nabla h(I_0)$, then using the isoenergetic non-degeneracy of $h$, (which is easily implied by the quasi-convexity assumption), it is enough to approximate $\omega_0$ by a periodic vector $\omega$. If $\Lambda$ has rank $n-1$, this is totally obvious, since necessarily we have $\Lambda=\omega^\perp \cap \Z^n$ for some periodic vector $\omega$ and so each action in $S_\Lambda$ is periodic. Now in the case where $\Lambda$ has rank $m=n-d$ with $d>1$, a good approximation is given by a theorem of Dirichlet, which moreover gives an explicit bound on the period $T$. 

\begin{proposition}\label{propCk2}
Let $I_0\in B_{R/2}$, $\Lambda$ be a sub-module of $\Z^n$ of rank $m=n-d$, with $d>1$ and $Q$ be a real number such that  
\begin{equation} \label{small1}
Q \PS 1. 
\end{equation}
Then there exists a $T$-periodic action $I_*\in B_R$ such that
\[ |I_0-I_*| \MP \max\left(d(I_0,S_\Lambda),T^{-1}Q^{-\frac{1}{d-1}}\right) \]
and the period $T$ satisfies
\[ 1\MP T \MP Q. \]
\end{proposition}

The proof of the above proposition can be found in~\cite{MS02}, Corollary $3.2$.

\paraga Now it remains to explain how the quasi-convexity hypothesis~(\ref{quasiconv}) on $h$, together with the normal form obtained in Proposition~\ref{lemmehamCk}, will enable us to control the variation $|I(t)-I_0|$, for an initial action $I_0$ close to some periodic action $I_*$. The idea goes as follows.

Let $\omega=\nabla h(I_*)$. We have already explained after Proposition~\ref{lemmehamCk} how in those new coordinates, the evolution of the actions in the direction generated by $\omega$ is controlled for some interval of time: if $F$ is the hyperplane orthogonal to $\omega$, then $\tilde{I}(t)$ remains close to the affine subspace $\tilde{I}_0+F$ for an interval of time $|t|\leq\tau$ where $\tau$ is essentially given by the inverse of the size of $\tilde{f}$. Now by preservation of energy, for all time $\tilde{I}(t)$ remains close to the unperturbed energy hypersurface $E_0=\{\tilde{I}\in B_R \;|\; h(\tilde{I})=h(\tilde{I}_0)\}$, and as the latter is strictly convex by quasi-convexity of $h$, the connected component of $\tilde{I}_0$ in $E_0 \cap (\tilde{I}_0+F)$ is in fact bounded and so is the variation $|\tilde{I}(t)-\tilde{I}_0|$ for $|t|\leq\tau$.

This idea is formalized in the proposition below. Once again, $I_*$ is a $T$-periodic action and $l$ is the linear Hamiltonian with periodic frequency $\omega=\nabla h(I_*)$.  

\begin{proposition}\label{propCk1}
Under the previous hypotheses, let $r>0$, $\tau>0$ and 
\[ \tilde{H}=h+g+\tilde{f} \in C^2(\T^n \times B(I_*,r))\] 
with $h$ satisfying~(\ref{quasiconv}), $\{g,l\}=0$ and the estimates
\[ |g+\tilde{f}|_{C^0(\T^n \times B(I_*,r))} < r^2, \quad |\partial_{\tilde{\theta}} \tilde{f}|_{C^0(\T^n \times B(I_*,r))} < r^2\tau^{-1}. \]
If
\begin{equation} \label{small3}
r \MP 1 
\end{equation} 
then for any initial condition $(\tilde{\theta}_0,\tilde{I}_0)\in \T^n \times B(I_*,r)$, the solution satisfies
\[ |\tilde{I}(t)-\tilde{I}_0| \MP r, \quad |t| \leq \tau. \]
\end{proposition}

Once again, we refer to \cite{MS02}, Corollary $3.1$, for a complete proof.

\paraga Let us now conclude the proof of Theorem~\ref{theoremeCk2}. In a first step we will use Proposition~\ref{propCk2} to find a periodic action close to our initial action, then in a second step we will apply Proposition~\ref{lemmehamCk} to find adapted coordinates and the third step will consist in applying Proposition~\ref{propCk1} to control the evolution of the action variables in those coordinates, and hence in the original coordinates. 

\begin{proof}[Proof of Theorem~\ref{theoremeCk2}]
Let $I_0\in B_{R/2}$, let $\Lambda$ be a sub-module of $\Z^n$ of rank $m=n-d$, and assume that $d(I_0,S_{\Lambda}) \MP \sqrt\varepsilon$ (note that this last assumption is void if $\Lambda$ is trivial).

\bigskip

{\it First step.} In the case $d=1$, any action $I\in S_{\Lambda}$ has a frequency $\nabla h(I)$ which is a multiple of some non zero vector $k_*\in\Z^n$, therefore we can choose a periodic action $I_*\in S_{\Lambda}$ so 
\[ |I_0-I_*|=d(I_0,S_{\Lambda}) \MP \sqrt\varepsilon \]
and the period $T$ trivially satisfies $T\MP 1$. In the case $d>1$, we apply Proposition~\ref{propCk2} with
\[ Q \EP \varepsilon^{-\frac{d-1}{2d}}. \]
and the condition~(\ref{small1}) gives a first smallness condition on $\varepsilon$. Observe that $Q^{-\frac{1}{d-1}} \EP \varepsilon^{\frac{1}{2d}}$, hence the periodic action $I_*$ given by the proposition satisfies
\[ |I_0-I_*| \MP \max\left(d(I_0,S_\Lambda),T^{-1}Q^{-\frac{1}{d-1}}\right) \MP T^{-1}\varepsilon^{\frac{1}{2d}} \] 
and the period $T$ 
\begin{equation}\label{estimT}
1\MP T \MP \varepsilon^{-\frac{d-1}{2d}}.
\end{equation}

\bigskip

{\it Second step.} Having found a periodic action, we will now apply Proposition~\ref{lemmehamCk} with 
\[ \mu \EP T^{-1}\varepsilon^{\frac{1}{2d}}. \]
With this choice, for $d>1$ the first part of condition~(\ref{small2}) is satisfied thanks to the upper bound~(\ref{estimT}) on the period $T$ (for $d=1$, this is trivial). The other conditions in~(\ref{small2}) give only further smallness conditions on $\varepsilon$. Applying the proposition, we have a $C^2$ symplectic transformation 
\[\Phi :\T^n \times B(I_*,\mu)\rightarrow \T^n \times B(I_*,2\mu)\] 
with $|\Pi_I\Phi-\mathrm{Id}|_{C^{0}(B(I_*,\mu))}\MP T\mu^2$ such that 
\begin{equation*}
H \circ \Phi=h+g+\tilde{f}
\end{equation*}
with $\{g,l\}=0$ and the estimates
\begin{equation*}
|g+\tilde{f}|_{C^{0}(\T^n \times B(I_*,\mu))} \MP \mu^2, \quad |\partial_{\tilde{\theta}} \tilde{f}|_{C^{0}(\T^n \times B(I_*,\mu))} \MP (T\mu)^{k-2}\mu^2. 
\end{equation*} 
Let us write $\tilde{H}=H \circ \Phi$ and $(\tilde{\theta},\tilde{I})$ are the new coordinates in $\T^n \times B(I_*,\mu)$.

\bigskip

{\it Third step.} Now we set
\[ r \EP \mu, \quad \tau \EP (T\mu)^{k-2},  \]
and we apply Proposition~\ref{propCk1} to the Hamiltonian $\tilde{H}$. To do so, we need to impose condition~(\ref{small3}) and this gives our last smallness condition on $\varepsilon$. Therefore we find
\[ |\tilde{I}(t)-\tilde{I}_0| \MP \mu, \quad |t| \MP (T\mu)^{k-2}, \]
and recalling that 
\[ \mu\EP T^{-1}\varepsilon^{\frac{1}{2d}} \MP \varepsilon^{\frac{1}{2d}}\] 
this gives
\[ |\tilde{I}(t)-\tilde{I}_0| \MP \varepsilon^{\frac{1}{2d}}, \quad |t| \MP \varepsilon^{\frac{k-2}{2d}}. \]
Now since
\[ |\Pi_I\Phi-\mathrm{Id}|_{C^{0}(B(I_*,\mu))}\MP T\mu^2 \MP \varepsilon^{\frac{1}{2d}}, \]
standard arguments give the conclusion
\[ |I(t)-I_0| \MP \varepsilon^{\frac{1}{2d}}, \quad |t| \MP \varepsilon^{\frac{k-2}{2d}}. \] 
This ends the proof.
\end{proof}

{\it Acknowledgements.} 
The author thanks Jean-Pierre Marco, Laurent Niederman and also H. Scott Dumas for a careful reading of a first version of the manuscript.

\addcontentsline{toc}{section}{References}
\bibliographystyle{amsalpha}
\bibliography{nekhoCk2}
\end{document}